\documentclass[12pt]{amsart}
\usepackage{amsmath}
\usepackage{amssymb}
\usepackage{amsfonts}
\usepackage{amsthm}
\usepackage{verbatim}
\usepackage{amscd}
\usepackage{cite}
\usepackage{leftidx}
\usepackage{enumerate}
\usepackage{txfonts}
\usepackage{manfnt}
\usepackage{amscd}
\usepackage[mathscr]{eucal}
\usepackage{hyperref}
\usepackage{palatino}

 \DeclareMathOperator{\gm}{\mathbb G_m}
\DeclareMathOperator{\del}{\partial}
 \DeclareMathOperator{\Hom}{Hom}

\def\refp #1.{(\ref{#1})}

\newcommand{\A}{\mathcal{A}}

\newcommand{\kk}{\mathbf{k}}

\def\sbr #1.{^{[#1]}}
\def\sfl #1.{^{\lfloor #1\rfloor}}

\def\?{{\bf{??}}}

\def\A{\Bbb A}

\def\P{\mathbb P}

\def\O{\mathcal O}

\def\g{\mathfrak g}

\def\1/2{\frac{1}{2}}

\def\2{{[2]}}
\def\l{\ell}
\def\nl{\newline}

\def\<{\langle}
\def\>{\rangle}

\def\2{{[2]}}
\def\l{\ell}

\def\scl #1.{^{\lceil#1\rceil}}
\def\spr #1.{^{(#1)}}
\def\sbc #1.{^{\{#1\}}}

\def\subpr#1.{_{(#1)}}

\def\beq{\begin{equation*}}
\def\eeq{\end{equation*}}

\def\g3{{\Gamma\spr 3.}}

\newcommand{\eqspl}[2]{
\begin{equation}\label{#1}
\begin{split}
#2\end{split}\end{equation}}
\newcommand{\eqsp}[1]{\begin{equation*}
\begin{split}#1\end{split}\end{equation*}}

\newcommand{\exseq}[3]{
0\to #1\to #2\to #3\to 0
}
\newcommand{\beginalphaenum}{
\begin{enumerate}\renewcommand{\labelenumi}{ }
\item \begin{enumerate}
}

\def\eex{\end{rm}\end{example}}

\newtheorem{thm}{Theorem} 

\newtheorem*{thm*}{Theorem}
\newtheorem*{prop*}{Proposition}
\newtheorem{cor}[thm]{Corollary}
\newtheorem*{cor*}{Corollary}

\newtheorem{lem*}{Lemma}

\newtheorem*{claim*}{Claim}
\newtheorem{prop}[thm]{Proposition}

\theoremstyle{remark}

\newtheorem{rem}[thm]{Remark}
\newtheorem{crit-rem}[thm]{Critical remark}
\newtheorem{remarks}[thm]{Remarks}

\newtheorem{example}[thm]{Example}
\newtheorem*{example*}{Example}

\newtheorem*{defn*}{Definition}
\def\fra{{\mathfrak a}}
\def\kk{\underline k}
\def\gm{\mathbb G_{\mathrm m}}
\begin{document} 
\title{Incidence stratifications on Hilbert schemes of smooth surfaces, and an application to Poisson structures}
\author 
{Ziv Ran}



\thanks{arxiv.org/1502.00553}
\date {\today}


\address {\nl UC Math Dept. \nl
Big Springs Road Surge Facility
\nl
Riverside CA 92521 US\nl 
ziv.ran @  ucr.edu\nl
\url{http://math.ucr.edu/~ziv/}
}

 \subjclass[2010]{14C05, 14J99, 32J15, 53D17}
\keywords{Complex surface, Hilbert scheme, stratification, normal crossings, Poisson structure, 
}

\begin{abstract}
Given a smooth curve on a smooth surface,  the Hilbert scheme of points on the surface 
is stratified according to the length of the intersection with the curve.
The strata are highly singular.
We show that this stratification admits a natural
log-resolution, namely the stratified blowup. As a consequence, the induced Poisson
structure on the Hilbert scheme of a Poisson surface has unobstructed deformations.
\end{abstract}
\maketitle
One of the important and well-studied geometric objects associated to a smooth
surface $X$ is the Hilbert scheme $X\sbr\l.$, parametrizing  0-dimensional
subschemes of length $\l$ on $X$.
This is a smooth $2\l$-dimensional variety, which inherits various aspects of
the geometry of $X$, e.g. a symplectic structure \cite{bvl}.
See \cite{lehn-montreal}, \cite{sernesi} for information and references on Hilbert schemes.\par
Now suppose one is interested not in the 'plain' surface $X$ but rather in a pair
 $(X, Y)$, where $Y$ is a smooth curve on $X$.
 To this one can quite analogously associate a {stratification}, called
 an \emph{incidence stratification}
\[Y\spr \l.=I^\l_Y\subset I^{\l-1}_Y\subset...\subset I^1_Y\subset X\sbr\l.\]
where the closed stratum $I^j_Y$ denotes the locus of schemes intersecting $Y$ in length at least $j$.
Though natural enough, this 
stratification unfortunately seems to have somewhat complicated singularities
except for the bottom stratum $I^\l_Y$: e.g. for $\l=2$, $I^1_Y$
has Whitney-umbrella type singularities along $I^2_Y$. 
Things get still more complicated in a neighborhood of worse-behaved,
e.g. non-curvilinear schemes.
Thus, one is led to try
to resolve the singularities of this stratification in a simple and natural way.\par
Given that $I^\l=Y\spr \l.$ 
is smooth, the simplest potential way to resolve the singularities of the
incidence stratification
is by 'stratified blowup': i.e. blow up $I^\l_Y$, then blow up the proper transform of
$I^{\l-1}_Y$, etc. The purpose of this paper is to show that the stratified blowup indeed resolves the
singularities of the incidence stratification.\par
This question, of independent interest, first arose in connection with Poisson structures
on Hilbert schemes of projective Poisson (i.e. anticanonical) surfaces (see \cite{qsymplectic},
\cite{lagrangian}), and indeed our result has some applications to such structures  
and their deformations, see Corollary \ref{poisson}.\par
Here we begin in \S\ref{uni} by proving an analogous stratification result for
loci of collections of univariate polynomials stratified by the number of their common zeros.
Then in \S\ref{monomial} we will prove the main result in a neighborhood of a \emph{monomial}
ideal. Finally in \S\ref{general} we will prove the general case by specializing a general ideal to a monomial one.
\par
In this paper we work over an algebraically closed field $\kk$ of arbitrary characteristic.\par
We heartily thank the referee for his careful reading of the paper and for many corrections and helpful comments.

\section{Univariate polynomials with many common zeros}\label{uni}
Our main theorem depends on a completely elementary result about a stratification
in certain spaces of polynomials in 1 variable, which corresponds to the number of 
common zeros of  polynomials.\par
Fix natural numbers  $m_1,...,m_n$ and set $m=m_1+...+m_n$. 
Consider the space $\A=\A(m_1,...,m_n)$
 of $2n$- tuples $(h_1,...,h_n, f_1, ...,f_n)$
of  polynomials (coefficients in $\kk$) of the form
\eqsp{
h_i&=x^{m_i}+\sum\limits_{j=0}^{m_i-1}b_{i,j}x^j,\\
f_i&=\sum\limits_{j=0}^{m_i-1}a_{i,j}x^j,\ \ \  i=1,...,n.
} This is clearly an affine space of dimension $2m$;
in fact, it can be identified with a linear space with origin $0=(x^{m_1},...,x^{m_n}, 0,...,0)$.
Set
\[h=\prod\limits_{i=1}^nh_i, p_i=f_i\prod\limits_{j\neq i}h_j.\]
Thus $h$ is monic of degree $m$ exactly and each $p_i$ is of degree at most $m-1$.
For each $k=m, m-1,...,1$, consider the locus $I^k\subset\A$
consisting of all $(h., f.)$ such that the ideal generated by $h, p_1,...,p_n$ has colength at least
$k$. Thus we have a chain of closed subschemes, i.e. a stratification
\[I^m\subset I^{m-1}\subset...\subset I^1\subset \A.\]
Consider the associated stratified blowup, i.e. the  blowup $\hat \A$ 
of $\A$ obtained by first blowing up $I^m$, then the proper transform of $I^{m-1}$,
etc.
\begin{prop}\label{nx-1var-prop}
Near the origin, for
 each $j=1,...,m$, the proper transform $\hat I^j$ of $I^j$ in $\hat A$
 is smooth and the total transform equals $\hat I^m+...+\hat I^j$
 and  is a divisor with normal crossings.

\end{prop}
\begin{proof}
To begin with, $I^m$ is defined by the vanishing of the $a_{i,j}, j=0,...,m_i-1, i=1,...,n$,
hence is smooth of codimension $m$. Hence the blowup $\A_1$ of $\A$ in $I^m$
is smooth, and is covered by open affines where some $a_{i,j}\neq 0$. Now on $\A_1$, 
the intersection of the exceptional divisor 
with the
proper
transform of $I^{m-1}$ is covered by open affines $U^1_i$ where $a_{i,m_i-1}\neq 0$
for some $i$, 
and on this open set, 
 the proper transform of $I^{m-1}$ is defined by the equations
\eqspl{vanish a}{a_{k, j}=0, \forall k\neq i, j=0,...,m_k-1,}
plus the equation
\eqspl{vanish f1}{f^1_{i}:=\mathrm{Rem}(h_i, f_i)=0}
where $\mathrm{Rem}(u,v)$ denotes remainder dividing $u$ by $v$;
thus in this case,
\[f_i^1=h_i-q_if_i, q_i:=\frac{1}{a_{i,m_i-1}}(x+b_{i,m_i-1}-a_{i,m_i-2}/a_{i,m_i-1}).\]
Note that the above formulas establish an isomorphism between the set of  pairs $(f_i, h_i) $ as above
(an open set  in an affine space) and the space of triples $(f_i^1, q_i, f_i)$ where 
\[\deg(f^1_i)\leq m_i-2, f_i=a_{i, m_i-1}x^{m_i-1}+\mathrm{(lower)},  q_i=\frac{1}{a_{i, m_i-1}}x+c; a_{i, m_i-1}\neq 0, 
c\in\kk,.\]
It follows in particular that the locus defined by $f^1_i=0$ is nonsingular of codimension $m_i-2$.
Write \[f^1_i=\sum\limits_{j=0}^{m^1_i-1} a^1_{i,j}x^j, m^1_i=m_i-1,\]
Also set 
\[f^1_{k}:= h_k, a^1_{k,j}=a_{k,j},  m^1_k=m_k, k\neq i.\]
Thus with \eqref{vanish a} and \eqref{vanish f1}
we get in all a total of $m-1=\sum\limits_{k=1}^n  m^1_k$ equations with independent differentials, 
namely
\[a^1_{k,j}=0, j=0,...,m_k^1-1, k=1,...,n,\]
defining the proper transform
of $I^{m-1}$, so this proper transform is smooth here and transverse
to the exceptional divisor, which has equation $a_{i, m_i-1}=0$.\par
Next we blow up the proper transform of $I^{m-1}$, thus obtaining $\A_2$
which is smooth, and has smooth exceptional divisor over $\A_1$.
Notice that the intersection of the exceptional divisor of $\A_2$ over $\A_1$
with the proper transform of $I^{m-2}$ is covered by open sets 
$U^1_iU^2_k$ where the leading ($m_i$-th) coefficient
of some $f_i$ and the leading  ($m^1_k$-th) coefficient
of some $f^1_k$ are both nonzero, $i,k=1,...,n$ distinct or not, and on $U^1_iU^2_k$ 
the proper
transform of $I^{m-2}$ is defined by the vanishing of (the coefficients of)
$f^2_k:=\mathrm{Rem}(h_k, f^1_k)$ (which is $m^2_k:=m^1_k-1$ many equations
with independent differentials), plus the vanishing
of $f^2_d:=f^1_d, \forall d\neq k$. Then we can continue in the same way.

\end{proof}
\section{Near a monomial ideal}\label{monomial}
In this section we will prove a special case of our main theorem
in a neighborhood of a monomial ideal cosupported at a point.\par
Let $(s.)=(s_1,...,s_n), (t.)=(t_1,...,t_n)$ be sequences of nonnegative integers. To these we associate
the planar ideal
\[\fra:=\fra(s., t.)=AG_0+ ...+AG_n,\ A:=\kk[x,y], \  G_i:=x^{t_1+...+t_{n-i}}y^{s_1+...+s_i},\ \  i=0,...,n.\]
We will assume this is cosupported at the origin, i.e. that $(\sum t_i)(\sum s_i)>0$. In this case
the ideal is of finite colength equal to
\[L=\sum\limits_{i+j\leq n+1} s_it_j.\]
As is often the case, $\fra$ can be usefully represented as a determinantal ideal, namely
as the ideal of  $n\times n$ minors of the $n\times(n+1)$ matrix $M=(m_{i,j}) $ where
\[m_{i,j}=\begin{cases} x^{t_i}, j=i;\\
y^{s_{n+1-i}}, j=i+1;\\
0, \mathrm{otherwise}.
\end{cases}
\]
Thus
\eqspl{}{M=\left [
\begin{matrix}
x^{t_1}&y^{s_n}&0&0&...&0\\
0&x^{t_2}&y^{s_{n-1}}&0&...&0\\
&&\vdots&&&\\
0&...&0&0&x^{t_n}&y^{s_1}
\end{matrix}
\right].
}
Thus $G_i=\det(M_i)$ where $M_i$ is the submatrix of $M$ obtained by deleting the $(n+1-i)$-th
column. Accordingly, $\fra$ admits the short resolution
\eqspl{resolution}{
0\to A^n\stackrel{^tM}{\to}A^{n+1}\stackrel{(G.)}{\to}\fra\to 0
}where $(G.)=(G_n, -G_{n-1},...,(-1)^nG_0)$.\par

Let $\mathcal H$ denote the Hilbert scheme 
of colength - $L$ ideals in $\kk[x,y]$, an open subset of the
Hilbert scheme of $\P^2$.
By Fogarty's theorem (see \cite{fogarty} or\cite{lehn-montreal} or \cite{sernesi}, 
Theorem 4.6.9, p.248), $\mathcal H$  is smooth at
the point corresponding to $\fra$, and has tangent space $\Hom(\fra, \kk[x,y]/\fra)$ so the latter
vector space has dimension $2L$. For a quick proof of Fogarty's theorem, use the exact sequence
\[\exseq{\fra}{\kk[x,y]}{\kk[x,y]/\fra}\]
to see that $\mathrm{Ext}^1(\fra, \kk[x,y]/\fra)\simeq \mathrm{Ext}^2(\kk[x,y]/\fra, \kk[x,y]/\fra)$, which is 
Serre dual to \nl $\mathrm{Hom}(\kk[x,y]/\fra, \kk[x,y]/\fra)$, hence is $L$-dimensional, and
morevoer that 
\[\chi(\fra, \kk[x,y]/\fra)=\dim(\Hom(\fra, \kk[x,y]/\fra))-\dim(\mathrm{Ext}^1(\fra, \kk[x,y]/\fra))\]
is locally constant on $\mathcal H$, being an Euler characteristic.\par
Now as is well known, local
deformations of $\fra$ are obtained by deforming
the matrix $M$, i.e. replacing $M$ by $\tilde M=M+N$, where $N$ can be taken with
coefficients in $(A/\fra)\otimes \mathfrak n$, where $S=\kk\oplus\mathfrak n$ is a local $\kk$-algebra.
In the case of infinitesimal deformations, $S$ is artinian, i.e.  finite-dimensional.
This well-known fact can be proved as follows. Given an $S$-flat ideal $\tilde\fra<A\otimes S$, lifting the
$G$ generators yields a map $\tilde G:A^{n+1}\otimes S\to\tilde\fra$. The kernel $\tilde K=\ker(\tilde G)$
is also $S$-flat and has $\tilde K\otimes(S/\mathfrak n)=\ker (G.)\simeq A^n$. Therefore $\tilde K$
itself is free so $\tilde K\simeq A^n\otimes S$.
\par
Let $t=t_1+...+t_n$ and note that $\fra.\kk[x]=\fra/\fra\cap (y)$ is an ideal of colength $t$.
For $k=1,...,t$ let $I^k_\fra$ denote the subscheme of $\mathcal H_\fra$, the germ of $\mathcal H$ at
$\fra$, consisting of deformations of $\fra$ whose image in $\kk[x]$ is of colength
at least $k$; the scheme structure on $I^k_\fra$ can be defined via a suitable Fitting ideal
associated to the restriction of \eqref{resolution} on the $x$-axis. Thus, we have a stratification by
closed subschemes
\eqspl{}{
I^t_\fra\subset I^{t-1}_\fra\subset...\subset I^1_\fra\subset \mathcal H_\fra.
}
Let $\hat{\mathcal H_\fra}$ denote the corresponding stratified blowup.
\begin{prop}\label{monomial-prop}
For each $k=t,...,1$, the proper transform $\hat I^k_\fra$ 
of $I^k_\fra$ in $\hat{\mathcal H_\fra}$ is smooth and the total
transform equals $\hat I^t_\fra+...+\hat I^k_\fra$ and is a divisor with  normal crossings.
\end{prop}
\begin{proof}
Write the main  diagonal and last column elements of $\tilde M$ as
\eqspl{}{
(\tilde M)_{i,i}=h_i=x^{t_i}+\sum b_{i,j}x^j, &(\tilde M_{i, n+1)}=f_i=\sum a_{i,j}x^j, i<n,\\
&(\tilde M_{n, n+1)}=y^{s_1}+f_n=y^{s_1}+\sum a_{n,j}x^j.
} Note that the deformations corresponding to $a_{i,j}x^j, b_{i,j}x^j$ are linearly independent
and together these yield a $2t$-dimensional subvariety of $\hat{\mathcal H_\fra}$,
which maps isomorphically to the space $\A$ considered in the previous sections.
Moreover, given a deformation of $\fra$ corresponding to $\tilde M$,
its restriction on the $x$-axis is determined by 
\[h=\prod h_i\]
which is the deformation of $G_0$, and by 
\[p_i=f_i\prod\limits_{j\neq i}h_j,\]
which is the deformation of $G_{n+1-i}, i=1,...,n$. Consequently, our result follows from
Proposition \ref{nx-1var-prop}.
\end{proof}

\section{Incidence stratifications: general case}\label{general}
We are now ready to state and prove the main result. Thus, let $Y$ be a smooth closed curve on the smooth
algebraic surface $X$ over an algebraically closed field $\kk$.
For simplicity, we (needlessly) assume $X$ quasi-projective, but see the remarks 
following the proof.
Let $I^k_Y\subset X\sbr \l.$ denote the subscheme of the length-$\l$ Hilbert scheme
of $X$ consisting of schemes whose intersection with $Y$ is of length $k$ or more. $I^k_Y$ may be endowed
with a scheme structure as the image of a natural closed subscheme of the flag Hilbert scheme
$X{\sbr k,\l.}$ (see \cite{sernesi}) that is the pullback of the closed subscheme $Y\sbr  k.\subset X\sbr k.$ under
the natural map $X\sbr k,\l.\to X\sbr k.$. We thus have a closed stratification, called the \emph{incidence
stratification} associated to $Y$:
\eqspl{}{
Y\spr \l.=I^\l_Y\subset I^{\l-1}_Y\subset...\subset I^1_Y\subset X\sbr \l..
} It is easy to see that each closed stratum $I^k_Y$ has codimension $k$ in $X\sbr \l.$.
\begin{thm}\label{mainthm}
In the stratified blowup of the incidence stratification, the proper transform $\hat I^k_Y $ of each
closed stratum $I^k_Y$ is smooth and the total
transform equals $\hat I^k_Y+\hat I^{k+1}_Y+...+\hat I^t_Y$ and has normal crossings.
\end{thm}
\begin{proof}
The statement is local near a given point $z\in X\sbr \l.$. Further, because of the usual \'etale, or analytic product
decomposition of the Hilbert scheme corresponding to the support of $z$, we may
assume $z$ is supported in a single point. Thus, we may work locally and assume $X$ is the plane $\A^2$ and $Y$
is the $x$-axis with ideal $(y)$ and $z$ is supported at the origin. Note that the 'punctual' Hilbert scheme
$X\sbr \l._0$ consisting of subschemes supported at the origin is projective. Consider the action of 
the multiplicative group $\gm$
on $X\sbr \l.$ and $X\sbr \l._0$ induced by the action on $X$ given by $y\mapsto \lambda.y$ (fixing $x$). 
Viewing $\gm$ as subset of $\A^1$, let
\[ z_0=\lim_{\stackrel{\lambda\in\gm}{\lambda\to 0}}\lambda^*z\in X\sbr\l._0.\]
In other words, the map
\[f_0:\gm\to X\sbr\l._0, f_0(\lambda)=\lambda^*z,\]
extends by projectivity to a map $f:\A^1\to X\sbr\l._0$ and $z_0=f(0).$
Then $z_0$ is $\gm$-invariant, i.e.  a homogeneous ideal with respect to $y$. 
Let $R=\O_{Y,0}$, which is  DVR with parameter $x$. Then the ideal $z_0R[y]$
in $R[y]$
is homogeneous as well with respect to $y$, hence is generated by finitely
many homogeneous elements of the form $ay^s, s\geq 0, a\in R$. Adjusting by
a unit, we may assume $a=x^r, r\geq 0$. Because $z_0=(z_0R[y])\cap \kk[x,y]$, it too is
generated by such elements, i.e. $z_0$ is
in fact a monomial ideal. By Proposition
 \ref{monomial-prop}, the Theorem holds in a neighborhood of $z_0$.
Moreover, every neighborhood of $z_0$ in $X\sbr\l.$ contains
schemes equivalent to $z$. Therefore the theorem holds locally near these as well, hence 
locally near $z$.
\end{proof}
\begin{remarks}
(i) Because Theorem \ref{mainthm} is local in nature, it actually holds without any quasi-projectivity hypotheses on $X$.
First, the Hilbert scheme $X\sbr \l.$ exists as a scheme (for any algebraic scheme $X$), 
and can be constructed as a projective morphism (cycle map)
over the symmetric product $X\spr\l.$, which itself can be covered by patches which are 
cartesian products of symmetric products of
quasi-projective schemes. 
See for instance \cite{structure}, \S1.
Second, the scheme structure on $I^k_Y$ can be constructed patch-wise
from the quasi-projective case. Lastly, the main argument of the proof is local around a scheme
supported on a single point, so certainly carries over.

\par
(ii) The obvious analogue of Theorem \ref{mainthm} in the complex-analytic category holds with the same proof, mutatis
mutandis.
\end{remarks}
As explained in \cite{qsymplectic} and in \cite{lagrangian}, Example 4.3, Theorem \ref{mainthm} has an
application to the deformation theory of induced 
Poisson structures on Hilbert schemes of Poisson
surfaces
(which in turn is an analogue of a result of Voisin \cite{voisin-lagrangian} in
the case of symplectic structures):
\begin{cor}\label{poisson}
Let $\Pi$ be a Poisson structure on a smooth complex projective surface $S$, corresponding to a
smooth anticanonical curve, and let  $\Pi\sbr r.$ be the associated Poisson structure
on the Hilbert scheme $S\sbr r.$. Then the pair $(S\sbr r., \Pi\sbr r.)$  has unobstructed deformations.
\end{cor}
\begin{proof}
In light of the argument in loc. cit. this almost  follows from normal crossings of $\hat I^r+...+\hat I^1$,
which coincides with the inverse image $\tilde D$ of the Pfaffian divisor of $\Pi$ on the stratified blowup
$\hat S\sbr r.$.
The only missing point is that $\Pi$ lifts holomorphically to $\hat S\sbr r.$. This can be checked locally 
at a generic point of each $\hat I^k$. Now it is clear from our proof above that such
a generic point corresponds to a \emph{reduced} scheme, i.e. $r$ distinct points on $S$, of which
exactly $k$ are on $C=[\Pi]$. There, $\Pi\sbr r.$ can be written locally as
\[y_1\del x_1\wedge\del y_1+...+y_k\del x_k\wedge\del y_k+\del x_{k+1}\wedge\del y_{k+1}+...
+\del x_r\wedge \del y_r.\]
On a suitable (typical) open set in the blowup, we can write locally $y_i=u_iy_1, i=2,...,k$
so clearly $\del y_i$ has a pole of at most $y_1$, which is cancelled by $y_i$, so
$\Pi\sbr r.$ is holomorphic.
\end{proof}
\vfill\eject
\bibliographystyle{amsplain}
\bibliography{../mybib}
\end{document}